\documentclass[11 pt, oneside, reqno]{amsart}
\usepackage[top=1.5in,bottom=1in,left=1in,right=1in]{geometry}
\usepackage{amsmath,amssymb,amsthm,amsopn,extarrows,accents,faktor,enumitem,stmaryrd,mathtools}  
\usepackage[linktocpage=true]{hyperref}
\hypersetup{
    colorlinks = true,
    allcolors = {MidnightBlue}}
\usepackage[dvipsnames]{xcolor}

\usepackage{setspace}
\usepackage{tikz-cd,tikz}
\usetikzlibrary{matrix,arrows,decorations.pathmorphing}

\newtheorem{lemma}[equation]{Lemma}
\newtheorem{theorem}[equation]{Theorem}
\newtheorem*{theorem*}{Theorem}
\newtheorem{proposition}[equation]{Proposition}
\newtheorem{corollary}[equation]{Corollary}
\theoremstyle{definition}
\newtheorem{definition}[equation]{Definition}
\newtheorem{remark}[equation]{Remark}
\newtheorem{example}[equation]{Example}

\numberwithin{equation}{section}

\newcommand{\too}{\longrightarrow}
\newcommand{\mtoo}{\longmapsto}
\newcommand{\htoo}{\lhook\joinrel\longrightarrow}

\newcommand{\g}{\mathfrak{g}}
\renewcommand{\b}{\mathfrak{b}}
\renewcommand{\d}{\mathfrak{d}}

\renewcommand{\t}{\mathfrak{t}}
\renewcommand{\c}{\mathfrak{c}}

\renewcommand{\l}{\mathfrak{l}}
\newcommand{\n}{\mathfrak{n}}
\renewcommand{\c}{\mathfrak{c}}

\newcommand{\uu}{\mathfrak{u}}
\newcommand{\z}{\mathfrak{z}}

\newcommand{\nuU}{\overline{U}}
\newcommand{\nB}{\overline{B}}
\newcommand{\nb}{\overline{\b}}
\newcommand{\nuu}{\overline{\uu}}

\renewcommand{\O}{\mathcal{O}}
\renewcommand{\S}{\mathcal{S}}
\newcommand{\C}{\mathcal{C}}

\newcommand{\I}{\mathcal{I}}
\newcommand{\A}{\mathcal{A}}

\DeclareMathOperator{\Ad}{Ad}
\DeclareMathOperator{\ad}{ad}

\makeatletter
\def\l@subsection{\@tocline{2}{0pt}{3.5pc}{5pc}{}}
\makeatother

\title{An analogue of Whittaker reduction for group-valued moment maps}
\author{Ana B\u{a}libanu}
\address{Department of Mathematics, Louisiana State University, Baton Rouge, LA  70803, USA}
\email{ana@math.lsu.edu}
\date{}

\begin{document}
\maketitle

\begin{abstract}
We construct an analogue of Whittaker reduction for Poisson actions of a semisimple complex Poisson--Lie group $G$. The reduction takes place along a class of transversal slices to unipotent orbits in $G$, which are generalizations of the Steinberg cross-section and are indexed by conjugacy classes in the Weyl group. We give an interpretation of these reductions in the framework of Dirac geometry, and we use this to describe their symplectic leaves.
\end{abstract}

\tableofcontents

%
%
%
%
%
%
%
\section*{Introduction}
Any semisimple complex Lie algebra $\g$ has a canonical Kirillov--Kostant Poisson structure, whose symplectic leaves are the orbits of the adjoint action. Each nilpotent element $f$ of $\g$ determines a Slodowy slice $\S$, introduced by Kostant \cite{kos:59} and Slodowy \cite{slo:80}, which is transverse to the adjoint orbits and strictly transverse to the nilpotent orbit containing $f$. Work of Gan and Ginzburg \cite{gan.gin:02} shows that the Kirillov--Kostant symplectic form on any adjoint orbit $\O$ restricts to a symplectic form on the intersection $\S\cap\O$, and the resulting symplectic foliation induces on $\S$ a natural Poisson structure.

Let $G$ be a semisimple complex group integrating $\g$, and let $M$ be a complex Poisson variety on which $G$ acts by Poisson diffeomorphisms. This action is called Hamiltonian if it is induced by a moment map
\[\mu:M\too \g.\]
Whittaker reduction is a type of Hamiltonian reduction, first defined by Kostant \cite{kos:79}, that takes place along $\mu$ at the nilpotent element $f$. It can be realized either as a symplectic reduction of $M$ with respect to the action of a maximal unipotent subgroup $U$ opposite to $f$, or as the preimage under $\mu$ of the Poisson transversal $\S$. These two constructions fit into the diagram
\begin{equation*}
\begin{tikzcd}[row sep=huge]
\mu^{-1}(\S)\arrow[r,hook]\arrow[rd,swap,"\sim"]	&\mu_U^{-1}(f)\arrow[r,hook]\arrow[d]	&M		\\
						&\mu_U^{-1}(f)/U,					&		
\end{tikzcd}
\end{equation*}
where $\mu_U$ is the moment map of the $U$-action, $f$ is viewed as an element of $\uu^*$ using the Killing form, and the diagonal map is an isomorphism. From this perspective, Whittaker reduction encodes the Poisson geometry of $M$ in the direction transversal to the orbits of $G$. 

\subsection*{Moment maps for Poisson actions} Poisson manifolds frequently exhibit natural symmetries that fail to preserve the Poisson bracket. It was observed by Semenov-Tian-Shansky \cite{sts:85} that this phenomenon is often explained by the fact that the group of symmetries itself carries an additional Poisson structure. In other words, the group $G$ which acts on the Poisson manifold $M$ is a Poisson–Lie group in the sense of Drinfeld \cite{dri:83}, and the action map is a Poisson map.

A Hamiltonian theory for such actions was developed by Lu \cite{lu:89}, and the associated moment maps take values in the dual Poisson–Lie group $G^*$. In the case when $G$ is a semisimple complex group equipped with the standard Poisson--Lie group structure, there is a local embedding 
\[G^*\cong B\times_T\nB\too G\]
whose image is the open dense Bruhat cell, and Lu's definition can be extended to consider moment maps valued in the group $G$ itself. This is a special case of the theory of $D/G$-valued moment maps introduced by Bursztyn and Crainic \cite{bur.cra:09}.

\subsection*{Slices to unipotent orbits} The multiplicative analogue of the Kostant slice was constructed by Steinberg \cite{ste:65} and is given by $U_ww,$ where $w$ is a minimal-length Coxeter element in the Weyl group of $G$ and $U_w$ is the subgroup generated by positive roots which are flipped by $w^{-1}$. This slice consists entirely of regular elements and is strictly transverse to the regular conjugacy classes in $G$. In the special case of $SL_n(\mathbb{C})$, it is precisely the space of Frobenius companion matrices. 

Steinberg's construction indicates that transverse slices to unipotent orbits in $G$ are linked to conjugacy classes in the Weyl group. Slices associated to non-Coxeter conjugacy classes have been studied by He and Lusztig \cite{lus.he:12}, by Sevostyanov \cite{sev:11}, and most recently by Duan \cite{dua:23}. While there are subtle technical differences between these constructions, they share two fundamental features. First, they associate to an element $w$ of the Weyl group a slice
\[\Sigma\coloneqq U_wZw,\]
where $U_w$ is defined as above, and $Z$ is the reductive subgroup of $G$ generated by $T^w$ and by the roots fixed by $w$. Second, each slice has the property that the conjugation map gives an isomorphism
\[U\times \Sigma\too UZwU\eqqcolon\Omega,\]
where $U$ is the unipotent subgroup consisting of positive roots not fixed by $w$. This isomorphism implies, in particular, that $\Sigma$ is transverse to the conjugacy classes of $G$.

The work of Sevostyanov shows that the slices $\Sigma$ have a natural Poisson structure inherited from the Poisson structure on $G$. More recently, the work of Duan makes the connection to unipotent orbits precise by showing that, when the conjugacy class of $w$ is ``most elliptic,'' the slice $\Sigma$ is strictly transverse to the unipotent orbit associated to $w$ under the Lusztig map \cite{lus:11, lus:12a, lus:12b}. Therefore, the slices $\Sigma$ exhibit multiplicative analogues of the key geometric features of Slodowy slices.

\subsection*{Summary of results} In the present work, we develop an analogue of Whittaker reduction for Poisson actions of Poisson--Lie groups, which takes place along the slices $\Sigma$. Concretely, equip the semisimple complex group $G$ with the standard Poisson--Lie group structure, and suppose that it has a Hamiltonian Poisson action on a complex manifold $M$ with corresponding moment map $\mu:M\too G$. There is a commutative diagram
\begin{equation*}
\begin{tikzcd}[row sep=huge]
\mu^{-1}(\Sigma)\arrow[r, hook]\arrow[rd,swap, "\sim"]	&\mu^{-1}(\Omega)\arrow[r, hook, "\jmath"]\arrow[d, "q"]	&M		\\
						&Q\coloneqq\mu^{-1}(\Omega)/U					&		
\end{tikzcd}
\end{equation*}
in which the diagonal map is an isomorphism. We prove the following main result as Theorem \ref{main}.

\begin{theorem*}
Let $\c$ be the orthogonal complement of the fixed-point set $\t^w$ in the maximal Cartan $\t$ of the Lie algebra $\g$. The quotient $Q$ carries a natural Poisson bracket $\{\cdot,\cdot\}_Q$ which is uniquely characterized by the property that
\[q^*\{f,g\}_Q=\jmath^*\{F, G\}\]
for all functions $f,g\in \O_Q$ and all $\c$-invariant lifts $F,G\in\O_M$ that satisfy $q^*f=\jmath^*F$ and $q^*g=\jmath^*G$.
\end{theorem*}

In the special case when $M$ is the group $G$ itself, the resulting Poisson structure on $\Sigma$ agrees with the Poisson structure originally introduced by Sevostyanov. Under the diagonal isomorphism, this Poisson structure can also be viewed as a pullback of the Poisson structure on $M$ in the sense of Dirac geometry. This allows us to characterize in Theorem \ref{main2} and Corollary \ref{main2cor} the symplectic leaves of a distinguished open dense subset $\Sigma^\circ$.

\begin{theorem*}
The symplectic leaves of
\[\mu^{-1}(\Sigma^\circ)\cong\mu^{-1}(\Omega)/U\]
are precisely the reductions of the symplectic leaves of $M$. In particular, $\mu^{-1}(\Sigma^\circ)$ is a clean Poisson--Dirac submanifold of $M$.
\end{theorem*}

\subsection*{Structure of the paper} In Section \ref{first} we review Slodowy slices and the definition of Whittaker reduction, and in Section \ref{second} we recall the theory of moment maps for Poisson actions of Poisson Lie--groups. In Section \ref{third} we survey constructions of transversal slices to the conjugation action, and we prove the main theorem as Theorem \ref{main}. Finally, in Section \ref{fourth} we interpret this theorem in the language of Dirac structures, and we use this to describe the symplectic leaves of the reduction.

\subsection*{Acknowledgments} The author would like to thank Chengze Duan, Sam Evens, Pedro Frejlich, Jiang-Hua Lu, and Alexey Sevastyanov for many interesting discussions. This work was partially supported by an NSF Standard Grant under award number DMS--2401514.

%
%
%
%
%
%
%
\section{Review of Whittaker reduction}
\label{first}
Let $G$ be a semisimple complex group and with Lie algebra $\g$. The Killing form $\kappa(\cdot,\cdot)$ is nondegenerate on $\g$, and gives a canonical identification between $\g$ and its dual $\g^*$. Pulling back the Kirillov--Kostant Poisson structure through this isomorphism, we obtain a Poisson structure on $\g$ whose symplectic leaves are the adjoint orbits. In this section we recall two equivalent definitions of Whittaker reduction, a type of Hamiltonian reduction that uses a class of affine slices which are transverse to the adjoint action.

\subsection{Slodowy slices}
Fix a nilpotent element $f$ of $\g$. By the Jacobson--Morozov theorem [Theorem 3.7.1]\cite{chr.gin:97}, it can be completed to an $\mathfrak{sl}_2$-triple $\{e,h,f\}$ which is unique up to the adjoint action of the centralizer $G^f$. The associated \emph{Slodowy slice} \cite{slo:80} is the affine space
\[\S\coloneqq f+\g^e,\]
which is strictly transverse to the adjoint orbit of the unipotent element $f$. Moreover, $\S$ intersects each adjoint orbit that it meets transversally and symplectically---that is, the restriction of the Kirillov--Kostant symplectic structure on any adjoint orbit $\O$ to the intersection $\S\cap\O$ is a symplectic form \cite[3.1]{gan.gin:02}. It follows that $\S$ is a Poisson transversal in $\g$ \cite{fre.mar:17}, and there is a unique Poisson structure on $\S$ whose symplectic leaves are the connected components of the intersections of $\S$ with the adjoint orbits.

The adjoint action of the semisimple element $h$ induces an eigenspace decomposition
\[\g=\bigoplus_{i\in\mathbb{Z}}\g_{(k)},\qquad\g_{(k)}\coloneqq\left\{x\in\g\mid [h,x]=kx\right\}\]
and the pairing
\[(x,y)\mtoo \kappa(f,[x,y]),\]
is nondegenerate on $\g_{(1)}$. If $\ell$ is a maximal isotropic subspace of $\g_{(1)}$ with respect to this pairing, the direct sum 
\[\n=\ell+\sum_{k\geq 2}\g_{(k)}\]
is a nilpotent subalgebra, integrated by a unipotent subgroup $N$ of $G$. The adjoint action defines an isomorphism
\begin{align}
\label{transv}
N\times \S &\xlongrightarrow{\sim} f+\n^\perp\\
(n,x)&\mtoo \Ad_n(x),\nonumber
\end{align}
where $\n^\perp$ is the annihilator of $\n$ under the Killing form \cite[Lemma 2.1]{gan.gin:02}. In other words, $\S$ is a cross-section to the adjoint action of $N$ on the affine space $f+\n^\perp$.

When all the eigenvalues of $h$ are even, the space $\n^\perp$ is the parabolic subalgebra spanned by the non-negative eigenspaces of $h$, and $\n$ is its nilradical. This occurs in particular when $f$ is regular---then $h$ and $e$ are also regular, the parabolic $\n^\perp$ is a Borel subalgebra, and the corresponding Slodowy slice is the Kostant slice \cite{kos:59}, which consists entirely of regular elements and meets every regular adjoint orbit in $\g$ in exactly one point.

%
%
%
%
%
%
%
\subsection{Whittaker reduction}
Suppose that $G$ has a Hamiltonian action on a Poisson manifold $M$, and that $\mu:M\too\g^*$ is the corresponding moment map. The restriction of the $G$-action to the subgroup $N$ is also Hamiltonian, and we get a commutative diagram of moment maps
\begin{equation*}
\begin{tikzcd}[column sep=large, row sep=huge]
M	\arrow[rd,swap,"\nu"]\arrow[r,"\mu"]	&\g^*\cong\g	\arrow[d, shift left=4.5]		\\
								&\,\quad\n^*\cong\g/\n^\perp		
\end{tikzcd}
\end{equation*}
where the identifications in the right column are made using the Killing form. The coset $[f]$ in $\g/\n^\perp$ is fixed by the action of $N$, and the \emph{Whittaker reduction} of $M$ with respect to $f$ is the Marsden--Weinstein reduction \cite{mar.wei:74} of $M$ at this fixed point. In other words, it is the quotient
\[\nu^{-1}([f])/N=\mu^{-1}(f+\n^\perp)/N,\]
equipped with the Marsden--Weinstein reduced Poisson structure. 

Since $\S$ is a Poisson transversal in $\g$, the preimage $\mu^{-1}(\S)$ is also a Poisson transversal in $M$ \cite[Lemma 7]{fre.mar:17}. In particular, it is a smooth submanifold which inherits a Poisson structure from $M$, and it sits in the commutative diagram
\begin{equation*}
\begin{tikzcd}[row sep=huge]
\mu^{-1}(\S)\arrow[r,hook]\arrow[rd,swap,"\sim"]	&\mu^{-1}(f+\n^\perp)\arrow[r,hook]\arrow[d]	&M		\\
						&\mu^{-1}(f+\n^\perp)/N.					&		
\end{tikzcd}
\end{equation*}
The diagonal map is an isomorphism by the transversality theorem \eqref{transv}, and it is Poisson because it is a composition of backward-Dirac maps. (See also Section \ref{fourth} for an overview of Dirac geometry.) Therefore the Whittaker reduction of $M$ at $f$ is isomorphic, as a Poisson manifold, to the preimage of the Slodowy slice $\S$ under the moment map.

%
%
%
%
%
%
%
\section{Poisson actions of Poisson--Lie groups}
\label{second}
We recall some general background on Poisson actions of Poisson--Lie groups in Sections \ref{2.1}, \ref{2.2}, and \ref{2.3}, and we refer to \cite{lu:90} and to the Appendix of \cite{eve.lu:07} for a more detailed exposition. In Section \ref{2.4} we discuss moment maps for Poisson actions, and we show that in the case of complex semisimple groups the usual definition of moment map admits a slight extension. Then, in Section \ref{2.5} we recall the basics of Poisson reduction and we prove a particular reduction procedure in the case of Poisson actions of Poisson Lie groups. 

\subsection{Poisson--Lie groups}
\label{2.1}
Let $G$ be any real or complex Lie group. A Poisson structure $\pi_G$ on $G$ is \emph{multiplicative} if the group multiplication
\[(G,\pi_G)\times (G,\pi_G)\too (G,\pi_G)\]
is a Poisson map, and in this case $(G,\pi_G)$ is called a \emph{Poisson--Lie group}. Because the multiplicative bivector $\pi_G$ vanishes at the identity, its differential induces a Lie cobracket
\[\delta:\g\too\g\wedge\g\]
that satisfies the cocycle condition
\[\delta\left([x,y]\right)=\ad_x\delta(y)-\ad_y\delta(x).\]
In particular, the cocycle condition implies that the dual
\[\delta^*:\g^*\wedge\g^*\too\g^*\]
defines a Lie bracket on $\g^*$. For this reason, the triples $(\g,[\cdot,\cdot],\delta)$ and $(\g^*,\delta^*,[\cdot,\cdot]^*)$ are called \emph{Lie bialgebras}. 

Conversely, a theorem of Drinfeld \cite{dri:83} shows that the simply-connected group integrating any Lie bialgebra is a Poisson--Lie group, with Poisson structure induced by the cobracket $\delta$. In particular, the simply-connected group $G^*$ integrating the Lie algebra $\g^*$ is equipped with a multiplicative Poisson structure $\pi_{G^*}$, and is called the \emph{dual Poisson--Lie group} of $(G, \pi_G)$.

\begin{remark}
\label{integration}
When $G^*$ is an integration of $\g^*$ which is not simply-connected, the cobracket $\delta$ may fail to integrate to a Poisson structure on $G^*$. However, this integration is guaranteed, for instance, whenever $\delta$ is given by an $r$-matrix, which is the case whenever $\g$ is reductive \cite{sts:83}. We will often use this implicitly in what follows.
\end{remark} 

Whenever $(\g,[\cdot,\cdot],\delta)$ is a Lie bialgebra, there is a unique Lie bracket on the vector space $\g\times\g^*$ with the property that 
\begin{itemize}[topsep=2.5pt, itemsep=2.5pt]
\item $\g$ and $\g^*$ are Lie subalgebras, and
\item the natural inner product
\[\langle (x,\alpha),(y,\beta)\rangle=\alpha(x)+\beta(y)\]
is invariant under the adjoint action of $\d$.
\end{itemize}
The product $\g\times\g^*$ equipped with this Lie bracket is called the \emph{double Lie algebra} of $(\g,[\cdot,\cdot],\delta)$ and is denoted
\[\d\coloneqq \g\bowtie\g^*.\]

\begin{example}
\label{additive}
The simplest example of a multiplicative Poisson structure on $G$ is the trivial structure $\pi_G=0$. In this case the induced Lie bracket on $\g^*$ is trivial, and the dual Poisson--Lie group is $\g^*$ itself, viewed as a group under addition and equipped with the usual Kostant--Kirillov Poisson structure. The double Lie algebra in this case is the semi-direct product $\g\ltimes\g^*$. 
\end{example}

The triple $(\d,\g,\g^*)$ is an example of a \emph{Manin triple}---a collection $(\d,\g,\g')$ consisting of 
\begin{itemize}[topsep=2.5pt, itemsep=2.5pt]
\item a Lie algebra $\d$ with an $\ad$-invariant inner product $\langle\cdot,\cdot\rangle$, and 
\item maximal isotropic Lie subalgebras $\g$ and $\g'$ that intersect trivially. 
\end{itemize}
The inner product then gives a natural identification of $\g'$ with the dual of $\g$, and the Lie bracket on $\g'$ gives $\g$ the structure of a Lie bialgebra.

\begin{example}
\label{standard}
Let $\g$ be a complex semisimple Lie algebra with Killing form $\kappa(\cdot,\cdot)$, and fix a pair $\b$ and $\nb$ of Borel subalgebras whose intersection is a Cartan subalgebra $\t$. Let $\uu$ and $\nuu$ be their respective nilradicals, and denote by $B$, $\nB$, and $T$ the corresponding subgroups of $G$.

The Lie algebra $\g\oplus\g$, whose Lie bracket is defined by letting the two $\g$-summands commute, has the natural invariant inner product
\[\langle (x_1,y_1),(x_2,y_2)\rangle=\kappa(x_1,x_2)-\kappa(y_1,y_2).\]
With respect to this inner product, the subalgebras
\[\g_\Delta\coloneqq\left\{(x,x)\in\d\mid x\in\g\right\}\]
and
\[\b\times_{\t}\nb\coloneqq\left\{(u+h,\bar{u}-h)\in\b\times\nb\mid u\in\uu,\, \bar{u}\in\nuu, \,h\in\t\right\}\]
are both Lagrangian, and their intersection is trivial. Because the Manin triple $(\g\oplus\g,\g_\Delta,\b\times_\t\nb)$ gives $\g_\Delta$ the structure of a Lie bialgebras, it induces a Poisson structure on $G$ which is called the \emph{standard Poisson--Lie group structure}. In view of Remark \ref{integration}, the corresponding dual Poisson--Lie group is the group 
\[G^*\cong B\times_T \nB\]
consisting of pairs of elements in $B$ and $\nB$ whose components in the maximal torus are inverse pairs.
\end{example}

%
%
%
%
%
%
%
\subsection{Poisson structures on the double}
\label{2.2}
Let $\{x_i\}$ be a basis of $\g$, let $\{\xi_i\}$ be the dual basis of $\g^*$, and consider the $2$-tensor
\[\Lambda\coloneqq \sum x_i\wedge\xi_i\in\wedge^2\d.\]
The bivector field
\[\pi_D^-\coloneqq\Lambda^R-\Lambda^L\]
is a multiplicative Poisson structure on the simply-connected group $D$ integrating the Lie algebra $\d$. This structure is the Poisson--Lie group structure corresponding to the natural isomorphism between $\d^*$ and $\d$ obtained by flipping the two factors, and the Poisson--Lie group $(D,\pi_D^-)$ is known as the \emph{Drinfeld double}. Both $(G,\pi_G)$ and $(G^*,\pi_{G^*})$ are Poisson--Lie subgroups of $(D,\pi_D^-)$.

The bivector field
\[\pi_D^+\coloneqq\Lambda^R+\Lambda^L\]
also defines a Poisson structure on $D$, known as the \emph{Semenov-Tian-Shansky} (\emph{STS}) \emph{Poisson structure}. The space $(D,\pi_D^+)$ is called the \emph{Heisenberg double}, and its symplectic leaves are the connected components of intersections of $(G,G^*)$ and $(G^*,G)$ double cosets \cite[Lemma 5.3]{eve.lu:07}. In particular, the Heisenberg double has an open dense symplectic leaf. In view of the following example, this space can be viewed as a generalization of the cotangent bundle $T^*_G$. 

\begin{example}
When $G$ carries the trivial Poisson--Lie group structure $\pi_G=0$ as in Example \ref{additive}, the double Lie algebra $\d=\g\ltimes\g^*$ is integrated by the group $D=G\ltimes\g^*$. In this case the STS Poisson structure $\pi_D^+$ on $G\ltimes\g^*$ agrees under the left-trivialization isomorphism
\[T^*_G\cong G\ltimes\g^*\]%
with the canonical symplectic structure on the cotangent bundle $T^*_G$.
\end{example}

Since the pushforward of $\Lambda^L$ under the quotient map
\[D\too D/G\]
is trivial, the two Poisson structures $\pi_D^\pm$ both descend to the same well-defined Poisson structure $\pi_0$ on the quotient $D/G$. The symplectic leaves of $\pi_0$ are the intersections of the orbits of the left actions of $G$ and $G^*$. Moreover, since the composition
\[\g^*\htoo\d\too\d/\g\]
is an isomorphism, it integrates to a local diffeomorphism
\[(G^*,\pi_{G^*})\too (D/G,\pi_0)\]
which is a Poisson map.

\begin{example} 
When $\pi_G=0$, the induced Poisson structure $\pi_0$ on the quotient
\[D/G\cong\g^*\]
is precisely the usual Kirillov--Kostant Poisson structure, and the map $\eqref{local}$ is an isomorphism.
\end{example}

\begin{example}
When $G$ is a semisimple complex group equipped with the standard Poisson--Lie group structure introduced in Example \ref{standard}, there is a canonical isomorphism
\begin{align}
\label{diagonal}
D/G_\Delta&\xlongrightarrow{\sim} G\\
(g_1,g_2)&\mtoo g_1g_2^{-1},\nonumber
\end{align}
This induces a local diffeomorphism
\[B\times_T\nB\too  G\]
whose image is the open dense Bruhat cell $B\nB$ of $G$. The symplectic leaves of $(G,\pi_0)$ are the intersections of conjugacy classes with orbits of $B\times_T\nB$ \cite[Proposition 2.9]{eve.lu:07}, which are all \cite[Lemma 2.10]{eve.lu:07} of the form
\[(B\times_T\nB)\cdot tw\qquad\text{for some}\quad t\in T \text{ and } w\in W.\]
In particular, each conjugacy class of $G$ contains a unique open dense symplectic leaf.
\end{example}

%
%
%
%
%
%
%
\subsection{Poisson actions}
\label{2.3}
An action of the Poisson--Lie group $(G,\pi_G)$ on a Poisson manifold $(M,\pi)$ is called \emph{Poisson} if the action map
\[(G,\pi_G)\times (M,\pi)\too (M,\pi)\]
is a Poisson map. When $\pi_G$ is $0$, this means that $G$ acts by Poisson automorphisms; however, when $\pi_G$ is nonzero, Poisson actions of $G$ generally do not preserve the Poisson structure $\pi$ of $M$.

Left and right multiplication define two Poisson actions the Poisson--Lie group $(D,\pi_D^-)$ on the Poisson manifold $(D,\pi_D^+)$. Moreover, the left multiplication map descends to a Poisson action of $(D,\pi_D^-)$ on the quotient $(D/G,\pi_0)$. Since $(G,\pi_G)$ is a Poisson--Lie subgroup of $(D,\pi_D^-)$, we therefore obtain an induced Poisson action of $(G,\pi_G)$ on $(D/G,\pi_0)$, which is called the \emph{dressing action}.

Under the local diffeomorphism \eqref{local}, the dressing action corresponds to an infinitesimal action of $\g$ on the dual Poisson--Lie group $G^*$ by \emph{dressing transformations}. These are defined by the map
\begin{align*}
\g&\too\mathfrak{X}(G^*) \\
\xi&\mtoo \pi_{G^*}^\#(\xi^L),
\end{align*}
where $\xi^L$ is the left-invariant 1-form on $G^*$ corresponding to the element $\xi$ of $\g$ under the identification $\g\cong(\g^*)^*.$

\begin{example}
In the case where $\pi_G=0$, the dressing action of $G$ on its dual Poisson--Lie group $\g^*$ is simply the coadjoint action.
\end{example}

\begin{example}
In the case where $G$ is a semisimple complex group equipped with the standard Poisson--Lie group structure defined in Example \ref{standard}, the infinitesimal dressing action of $\g$ on the dual Poisson--Lie group $B\times_T\nB$ does not integrate to an action of $G$ on $B\times_T\nB$. However, under the isomorphism \eqref{diagonal} the dressing action of $G$ on $D/G_\Delta\cong G$ is simply the conjugation action.
\end{example}

%
%
%
%
%
%
%
\subsection{Moment maps for the dressing action}
\label{2.4}
Let $(M,\pi)$ be a Poisson manifold. Any Poisson map $\mu:(M,\pi)\too (G^*,\pi_{G^*})$ induces an infinitesimal action of $G$ on $M$ given by
\begin{align}
\label{induce}
\g&\too\mathfrak{X}(M) \\
\xi&\mtoo \pi^\#(\mu^*(\xi^L))\nonumber
\end{align}
where once again $\xi^L$ is the left-invariant $1$-form on $G^*$ generated by the element $\xi$ of $\g$. The map $\mu$ intertwines this action with the infinitesimal dressing action of $\g$ on $G^*$. If the infinitesimal action \eqref{induce} integrates, the resulting action of $G$ on $M$ is Poisson \cite[Theorem 4.8]{lu:90}, and $\mu$ is called its \emph{moment map} \cite[Definition 4.1]{lu:90}. Conversely, a Poisson action of $(G,\pi_G)$ on $M$ is called \emph{Hamiltonian} if it is induced by a Poisson map $\mu:(M,\pi)\too (G^*,\pi_{G^*})$ in this way.

\begin{example}
When $\pi_G=0$, Poisson actions of $G$ are exactly actions of $G$ by Poisson diffeomorphisms, and a moment map for such an action is precisely the usual notion of moment map valued in the dual Poisson--Lie group $\g^*$.
\end{example}

From now on suppose that $G$ is a semisimple complex group equipped with the standard Poisson--Lie group structure as in Example \ref{standard}. In this case, the non-integrability of the infinitesimal dressing action may prevent the induced action of $\g$ on $M$ from integrating. However, we can enlarge the class of moment maps we consider by extending their codomain along the map 
\begin{align}
\label{local}
\lambda:B\times_T\nB&\too  G\\
	(b_1,b_2)&\mtoo b_1b_2^{-1}, \nonumber
	\end{align}
allowing them to take values in $G$.

\begin{lemma}
\label{forms}
For each element $\xi$ of $\g$, there is a 1-form $\alpha_\xi\in\Omega^1(B\nB)$ such that $\lambda^*\alpha_\xi=\xi^L$.
\end{lemma}
\begin{proof}
The group $B\times_T\nB$ acts on itself by left multiplication and on $G$ by 
\[(b_1,b_2)\cdot g=b_1gb_2^{-1},\]
the map \eqref{local} is equivariant with respect to these actions. Since $B\nB$ is a single $B\times_T\nB$-orbit in $G$, we define $\alpha_\xi$ at the identity by
\begin{align*}
\alpha_\xi:\g&\too\mathbb{C}\\
		x&\too\kappa(\xi,x)
		\end{align*}
and then extend it to $B\nB$ using the $B\times_T\nB$-action. Then $\lambda^*\alpha_\xi=\xi^L$ at the identity, and therefore by invariance equality holds at all points of $B\nB$.
\end{proof}

\begin{definition}
\label{defn}
Suppose that $(M,\pi)$ is a Poisson manifold with a Poisson action of the complex semisimple Poisson--Lie group $(G,\pi_G)$, and let
\[\rho_M:\g\too T_M\]
be the associated infinitesimal action. A $G$-equivariant map $\mu:M\too G$ is a \emph{moment map} for this Poisson action if the equality
\[\rho_M(\xi)=\pi^\#(\mu^*\alpha_\xi)\]
holds along the dense open subset $M^\circ=\mu^{-1}(B\nB)$.
\end{definition}

The moment maps of Definition \ref{defn} are examples of a more general theory of \emph{$D/G$-valued moment maps} \cite{bur.cra:09}. In particular, since the restriction
\[\mu:M^\circ\too B\nB\]
is a Poisson map between the dense open sets $M^\circ$ and $B\nB$, it follows that $\mu:(M,\pi)\too(G,\pi_G)$ is also Poisson.

\begin{example}
The Heisenberg double $(D,\pi_D^+)$ has two Poisson actions of $(D,\pi_D^-)$, given by left- and right-multiplication. Since the image of the diagonal inclusion
\[(G,\pi_G)\htoo (D,\pi_D^-)\]
is a Poisson--Lie subgroup of $(D,\pi_D^-)$, we obtain two commuting Poisson actions of $(G,\pi_G)$ defined by
\begin{align*}
(G\times G)\times D &\too D\\
(g,h)\cdot (u,v)&\mtoo (guh^{-1}, hvg^{-1}).
\end{align*}
The moment map corresponding to this Poisson action in the sense of Definition \ref{defn} takes values in the space $G\times G$ equipped with the STS Poisson structure $\pi_0\times\pi_0$, and is given by
\begin{align*}
D&\too G\times G\\
(u,v)&\mtoo (uv, u^{-1}v^{-1}).
\end{align*}

We will often consider the reparametrization of $D$ defined by $(a,b)\mapsto(u,vu)$, which is a multiplicative analogue of the usual left-trivialization of the cotangent bundle of $G$. Under this change of coordinates, the $G\times G$-action becomes 
\begin{align*}
(G\times G)\times D &\too D\\
(g,h)\cdot (a,b)&\mtoo (gah^{-1}, hbh^{-1}),
\end{align*}
and the moment map is given by 
\begin{align*}
\mu:\,\,\,\,D\,\,\,\,&\too G\times G\\
(a,b)&\mtoo (aba^{-1}, b^{-1}).
\end{align*}
\end{example}

%
%
%
%
%
%
%
\subsection{Poisson reduction}
\label{2.5}
In this section we begin by recalling the general notion of Poisson reduction, and we refer to \cite{zam:11} for a more detailed overview. Let $(M, \pi)$ be a complex Poisson manifold with associated bracket
\[\{\cdot,\cdot\}:\O_M\times\O_M\too \O_M,\]
let $\imath:S\htoo M$ be a submanifold, and write
\[\I=\left\{f\in \O_M\mid f_{\vert S}=0\right\}\]
for its vanishing ideal. Then $\I$ is a Poisson ideal in $\O_M$ if and only if $S$ is a Poisson submanifold of $M$, and in this case the Poisson structure on $S$ is precisely the bracket $\{\cdot,\cdot\}_S$ induced on the quotient
\[\O_S=\O_M/\I.\]
It is uniquely characterized by the property that, for any $f,g\in \O_S$ and any extensions $F,G\in \O_M$ such that $f=\imath^*F$ and $g=\imath^*G,$
\[\{f,g\}_S=\imath^*\{F,G\}.\]

More generally, the ideal $\I$ is a Poisson subalgebra of $\O_M$ if and only if $S$ is a coisotropic submanifold---that is, if and only if
\[\pi^\#(T_S^\circ)\subset T_S.\]
The Poisson normalizer
\[N(\I)=\left\{h\in \O_M\mid \{h,\I\}\subset \I\right\}\]
consists of those functions in $O_M$ which are invariant along the distribution $\pi^\#(T^\circ_S)$. It is a Poisson subalgebra of $\O_M$, and $\I$ sits inside $N(\I)$ as a Poisson ideal. The quotient $N(\I)/\I$ is therefore a Poisson algebra. 

If the distribution $\pi^\#(T_S^\circ)$ is regular, it integrates to a regular foliation of $S$ and $N(\I)/\I$ is the subalgebra of $\O_S$ consisting of functions that are constant along the leaves of this foliation. The leaf space $Q\coloneqq S/\pi^\#(T_S^\circ)$ is therefore a Poisson manifold which sits in the diagram
\begin{equation*}
\begin{tikzcd}[row sep=large]
S\arrow[r, hook, "\imath"]\arrow[d, swap, "q"]	&M		\\
						Q.		
\end{tikzcd}
\end{equation*}
The associated Poisson bracket $\{\cdot,\cdot\}_Q$ is uniquely characterized by the property that, for any $f,g\in \O_Q$ and any lifts $F,G\in \O_M$ such that $q^*f=\imath^*F$ and $q^*g=\imath^*G,$
\[q^*\{f,g\}_Q=\imath^*\{F,G\}.\]
This construction of this Poisson structure is called the \emph{coisotropic reduction} of $M$ along $S$.

When the multiplicative ideal $\I$ is not a Poisson ideal, we can consider instead the intersection $N(\I)\cap\I$, which is a Poisson ideal of $N(\I)$. The resulting Poisson algebra 
\[N(\I)/(N(\I)\cap\I)\]
is precisely the subalgebra of $\O_S$ consisting of functions which admit extensions to $M$ that are invariant along $\pi^\#(T_S^\circ)$. This ring is therefore the candidate for a ``reduction'' of the Poisson structure of $M$ along the submanifold $S$.

When $\pi^\#(T_S^\circ)\cap T_S$ is a regular distribution, the quotient $Q\coloneqq S/(\pi^\#(T_S^\circ)\cap T_S)$ is a manifold and
\[\O_Q=\A/\I,\]
where
\[\A=\left\{h\in \O_M\mid \{h,N(\I)\cap\I\}\subset\I\right\}\]
is the subalgebra of $\O_M$ consisting of functions which are invariant along $\pi^\#(T^\circ_S)\cap T_S$. Since $\A$ is generally strictly larger than $N(\I)$, there is a priori only an inclusion
\[N(I)/(N(\I)\cap\I)\subset\A/\I\]
and the quotient $Q$ does not necessarily inherit a Poisson structure. However, equality holds under certain conditions, such as when the distribution $\pi^\#(T_S^\circ)+T_S$ has constant rank, in which case $S$ is called a \emph{pre-Poisson submanifold} \cite{cat.zam:07}.

Based on this framework, we prove a special case of Poisson reduction involving Poisson actions of Poisson--Lie groups, which we will apply in Section \ref{third}.

\begin{proposition}
\label{reduction}
Let $(M,\pi)$ be a Poisson manifold with Poisson bracket $\{\cdot,\cdot\}$ and equipped with a Poisson action of a Poisson--Lie group $H$. Suppose that $\imath:S\htoo M$ is an $H$-stable submanifold such that
\begin{itemize}
\item the quotient $Q\coloneqq S/H$ is a manifold with quotient map $q:S\too Q$,
\item the characteristic distribution
\[\pi^\#(T_S^\circ)\cap T_S\]
is tangent to the orbits of $H$, and
\item every $H$-invariant function on $S$ admits a $\pi^\#(T_S^\circ)$-invariant extension to $M$.
\end{itemize}
Then $Q$ inherits from $M$ a natural Poisson structure $\{\cdot,\cdot\}_Q$ which is uniquely characterized by the property that
\begin{equation}
\label{characterization}
q^*\{f,g\}_Q=\imath^*\{F,G\}
\end{equation}
for all $f,g\in \O_Q$ and all $\pi^\#(T_S^\circ)$-invariant lifts $F,G\in \O_M$ that satisfy $q^*f=\imath^*F$ and $q^*g=\imath^*G$.
\end{proposition}
\begin{proof}
Once again let $\I$ be the vanishing ideal of the submanifold $S$. The group $H$ preserves both $\pi^\#(T_S^\circ)$ and $T_S$, and therefore acts on the rings 
\[N(I)/(N(\I)\cap\I)\subset\A/\I\subset \O_M/\I=\O_S,\]
where $\A$ consists, as above, of $(\pi^\#(T_S^\circ)\cap T_S)$-invariant functions in $\O_M$. 

Taking invariants with respect to the $H$-action, we obtain
\[\left(N(I)/(N(\I)\cap\I)\right)^H=\left(\A/\I\right)^H=\left(\O_M/\I\right)^H=\O_S^H=\O_Q.\]
Here the first equality follows from the third assumption, and the second equality from the second. In particular, since the left-most side of this equality is a Poisson algebra, $\O_Q$ carries a reduced Poisson structure uniquely characterized by the property \eqref{characterization}.
\end{proof}

%
%
%
%
%
%
%
\section{Groupy Whittaker}
\label{third}
From now on we let $G$ be a complex semisimple algebraic group. In this section we study Poisson reduction at a class of smooth subvarieties in $G$ which are transversal slices to the conjugation action. We begin by recalling the construction of these slices in Section \ref{slices1}, and we compute their characteristic distributions in Section \ref{slices3}. In Section \ref{slices4} we prove the main reduction theorem.

\subsection{Multiplicative transversal slices}
\label{slices1}
Fix a maximal torus $T$ of $G$ and let $W=N_G(T)/T$ be the associated Weyl group. Any conjugacy class of elements in $W$ determines a slice in $G$ which is transverse to the orbits conjugation action. Such slices have been defined by several authors \cite{ste:65, lus.he:12, sev:11, dua:23}, and here we follows the conventions of \cite{dua:23}.

Fix a conjugacy class in $W$. For a particular choice \cite[Definition 4.3]{dua:23} of element $w$ in this conjugacy class, we define $U$ to be the unipotent subgroup generated by the roots not fixed by $w$,
\[Z\coloneqq \left\langle T^w, \,U_\alpha\mid w(\alpha)=\alpha\right\rangle,\qquad\text{and}\qquad U_w\coloneqq U\cap w\nuU w^{-1},\]
where for convenience we abuse notation to also denote by $w$ the normal representative of this Weyl group element in $N_G(T)$. Then $Z$ centralizes $w$, and $U_w$ is the unipotent subgroup consisting of positive roots that are flipped by $w^{-1}$. 

The slice associated to $w$ is the smooth subvariety
\[\Sigma\coloneqq U_wZw.\]
It is transverse to each conjugacy class of $G$ that it meets, and it is strictly transverse to the $U$-orbits on the enlarged space
\[\Omega\coloneqq UZw U=U_wZw U\]
---that is, the conjugation map
\begin{align}
\label{transverse2}
U\times \Sigma&\too U_wZw U\\
(u,h)&\mtoo uhu^{-1}\nonumber
\end{align}
is an isomorphism of varieties \cite[Theorem 1.2]{dua:23}.

\begin{remark}
Since the subgroup $Z$ may be disconnected in general, the slice $\Sigma$ is not necessarily connected. In the special case where $w$ is a Coxeter element, the subgroup $Z=T^w$ is a finite group and the identity component of the slice $\Sigma$ is precisely the Steinberg cross-section 
\[U_ww\]
introduced in \cite{ste:65}, which consists entirely of regular elements and is strictly transverse to each regular conjugacy class. This cross-section is a group counterpart of the Slodowy slice corresponding to a regular nilpotent element.
\end{remark}

\begin{remark}
Steinberg's original construction indicates that transversal slices to unipotent orbits in the group $G$ are linked to conjugacy classes in the Weyl group. This relationship, which motivates the results of \cite{lus.he:12} and \cite{dua:23}, is made explicit through the Lusztig map \cite{lus:11, lus:12a, lus:12b}, which we now briefly recall.

Given a conjugacy class $\C$ in $W$ and any element $w\in\C$ of minimal length, there is a unique minimal unipotent orbit $\O_{\C}$ with the property that the intersection 
\[BwB\cap\O_{\C}\]
is nonempty, and every unipotent orbit arises in this way \cite[Theorem 0.4]{lus:11}. The resulting \emph{Lusztig map} is the surjective assignment
\begin{align*}
\Phi:\{\text{conjugacy classes in $W$}\}&\too\{\text{unipotent orbits in $G$}\}\\
				\mathcal{C}&\mtoo\O_\C.
\end{align*}

This map is not injective in general, but becomes injective when it is restricted to the set of \emph{elliptic} conjugacy classes---conjugacy classes of Weyl group elements which fix no points in the maximal Cartan $\t$ \cite[Proposition 0.6]{lus:11}. He and Lusztig \cite {lus.he:12} showed in arbitrary characteristic that, whenever $\C$ is an elliptic conjugacy class and $w\in \C$ is a minimal length representative in good position, the space $\Sigma$ satisfies the transversality property \eqref{transverse2}. Around the same time, Sevostyanov \cite{sev:11} studied the slices $\Sigma$ for arbitrary choices of $w$ in characteristic $0$, and proved that they satisfy \eqref{transverse2} and that they are transversal to the conjugation action. 

Recently Duan \cite{dua:23} generalized these results to arbitrary characteristic and showed that, when $\C$ is the ``most elliptic'' class in the preimage $\Phi^{-1}(\O_{\C})$ of its associated unipotent orbit \cite[Theorem 0.2]{lus:12a}, \cite[Theorem 1.16]{lus:12b} and $w$ is an element of $\C$ which has a good position representative in the braid group of $G$, the slice $\Sigma$ is of complementary dimension to $\O_{\C}$ and therefore strictly transverse to it \cite[Theorem 1.2]{dua:23}. 

There are subtle differences between the approaches of He--Lusztig and Duan and that of Sevostya-nov regarding the choice of the element $w$. The relationship between these constructions has been explored in \cite{mal:21}. However, these subtleties do not play any role in our results, which rely only on the transversality property \eqref{transverse2} that is a feature of all the constructions mentioned above.
\end{remark}

%
%
%
%
%
%
%
\subsection{Characteristic distributions in $(G,\pi_0)$}
\label{slices3}
Let $\Phi$ be the root system of $\g$ with respect to the maximal Cartan $\t$, and let $\Phi^+$ be the set of positive roots determined by the Borel $\b$. There is a basis of $\g$ denoted by
\[\left\{h_1,\ldots,h_l, e_{\pm\alpha}\mid h_i\in\t, \alpha\in\Phi^+\right\},\]
which has the property that
\begin{itemize}[topsep=2.5pt, itemsep=2.5pt]
\item $2\kappa(h_i,h_j)=\delta_{ij}$ and 
\item $e_{\pm\alpha}$ is a root vector of weight $\pm\alpha$ that satisfies $\kappa(e_\alpha,e_{-\alpha})=1.$
\end{itemize}
With respect to this basis \cite[2.9]{eve.lu:07}, the Poisson structure $\pi_0$ is given by the bivector
\begin{equation}
\label{bivector}
\pi_0=\sum_{i}h_i^L\wedge h_i^R
		+\frac{1}{2}\sum_{\alpha\in\Phi^+}\left(e_\alpha^L\wedge e_{-\alpha}^L
		+e_{\alpha}^R\wedge e_{-\alpha}^R\right)
		+\sum_{\alpha\in\Phi^+}e_{-\alpha}^L\wedge e_{\alpha}^R.
\end{equation}
Let $\c$ be the orthogonal complement to $\t^w$ in $\t$ with respect to the Killing form. Then $\c$ is stable under the action of $w$ and contains no $w$-fixed points, so
\[\l=\z\oplus\c\qquad\text{and}\qquad\c\subset Z(\l),\]
where $\l$ is the Levi subgroup corresponding to the root subsystem $\{\alpha\in\Phi\mid w(\alpha)=\alpha\}$. The subalgebra $\c$ is integrated by a subtorus $C$ of the maximal torus $T$.

\begin{proposition}
\label{sharp}
The distribution $\pi^\#(T_\Omega^\circ)$ is tangent to the orbits of the subgroup $UC$.
\end{proposition}
\begin{proof}
The action of $G$ on itself by conjugation induces an infinitesimal action map
\begin{align*}
\rho:\g&\too T_G\\
	\xi&\mtoo \xi^L-\xi^R.
	\end{align*}
The proposition therefore amounts to showing that $\pi_0^\#(T_\Omega^\circ)$ is contained in $\rho(\uu+\c).$

Using the Killing form to identify $T_G$ and $T_G^*$, we have
\[T_{\Omega}^\circ=\left((\uu+\z)^L+(\uu+\z)^R\right)^\circ=(\uu+\c)^L\cap(\uu+\c)^R.\]
If $x$ is an element of $T_{\Omega}^\circ$, using \eqref{bivector} and orthogonality we simplify each summand of $\pi_0^\#(x)$ to obtain
\begin{align*}
&\left(h_i^L\wedge h_i^R\right)^\#(x)=\kappa\left(h_i^L, x\right)h_i^R-\kappa\left(h_i^R, x\right)h_i^L\\
&\left(e_\alpha^L\wedge e_{-\alpha}^L\right)^\#(x)=-\kappa\left(e_{-\alpha}^L, x\right)e_{\alpha}^L\\
&\left(e_\alpha^R\wedge e_{-\alpha}^R\right)^\#(x)=-\kappa\left(e_{-\alpha}^R, x\right)e_{\alpha}^R\\
&\left(e_{-\alpha}^L\wedge e_{\alpha}^R\right)^\#(x)=\kappa\left(e_{-\alpha}^L, x\right)e_{\alpha}^R.
\end{align*}
Writing $x=n_1^L+c_1^L=n_2^R+c_2^R$ for some elements $n_1, n_2$ of $\uu$ and $c_1, c_2$ of $\c$, the first sum becomes
\begin{align*}
\sum\kappa\left(h_i^L, x\right)h_i^R-\kappa\left(h_i^R, x\right)h_i^L=\sum\kappa\left(h_i, c_1\right)h_i^R-\kappa\left(h_i, c_2\right)h_i^L=\frac{1}{2}c_1^R-\frac{1}{2}c_2^L
			\end{align*}
since $\t$ is orthogonal to $\uu$. The second summation gives
\[\sum\kappa\left(e_{-\alpha}^L, x\right)e_{\alpha}^L
		=\sum\kappa\left(e_{-\alpha}, n_1\right)e_{\alpha}^L
		=n_1^L\]
since $\kappa(e_\alpha,e_{\beta})=0$ when $\beta\neq-\alpha$. Similarly, the third summation simplifies to
\[\sum\kappa\left(e_{-\alpha}^R, x\right)e_{\alpha}^R
		=\sum\kappa\left(e_{-\alpha}, n_2\right)e_{\alpha}^R
		=n_2^R.\]
Finally, the third summation gives
\[\sum\kappa\left(e_{-\alpha}^L, x\right)e_{\alpha}^R
		=\sum\kappa\left(e_{-\alpha}, n_1\right)e_{\alpha}^R
		=n_1^R.\]
We obtain
\[\pi_0^\#(x)=\frac{1}{2}\left(c_1^R-c_2^L\right)-\frac{1}{2}(n_1^L+n_1^L+c_1^L-c_2^R)+n_1^R=\frac{1}{2}\left(\rho(-c_1)+\rho(c_2)\right)+\rho(-n_1),\]
completing the proof.
\end{proof}

\begin{corollary}
\label{char}
The characteristic distribution of the submanifold $\Omega$ is tangent to the orbits of the unipotent radical $U$.
\end{corollary}
\begin{proof}
Using Proposition \ref{sharp}, it suffices to show that
\[\rho(\uu+\c)\cap T_\Omega=\rho(\uu).\]
Suppose towards a contradiction that $c$ is an element of $\c$ with the property that $c^L-c^R\in T_\Omega$ at some point $g$ of $\Omega$. It follows that
\[c-\Ad_g(c)\in\uu+\z+\Ad_g(\uu)\]
and, simplifying the left-hand expression, we see that $c=\Ad_w(c).$ But $w$ has no nontrivial fixed points on $\c$, so $c=0$.
\end{proof}

%
%
%
%
%
%
%
\subsection{Reduction along the slice $\Sigma$}
\label{slices4}
Now let $(M,\pi)$ be a smooth Poisson algebraic variety with a Hamiltonian Poisson action of the Poisson--Lie group $(G,\pi_G)$. Write
\[\rho_M:\g\too T_M\]
for the corresponding infinitesimal action, and let
\[\mu:M\too G\]
be the associated moment map in the sense of Definition \ref{defn}. Since $\Sigma$ is transverse to the conjugacy classes it meets, so is $\Omega$. Then, since the moment map $\mu$ is $G$-equivariant, it follows that $\Sigma$ and $\Omega$ are transverse to $\mu$. In particular, their preimages are smooth subvarieties of $M$, and we obtain a commutative diagram
\begin{equation*}
\begin{tikzcd}[row sep=huge]
\mu^{-1}(\Sigma)\arrow[r, hook]\arrow[rd,swap]	&\mu^{-1}(\Omega)\arrow[r, hook, "\jmath"]\arrow[d, "q"]	&M		\\
						&Q\coloneqq\mu^{-1}(\Omega)/U.					&		
\end{tikzcd}
\end{equation*}
The diagonal map is an isomorphism by the transversality property \eqref{transverse2}. In particular, the quotient $Q$ always exists.

\begin{theorem}
\label{main}
There is a natural Poisson bracket $\{\cdot,\cdot\}_Q$ on the quotient $Q$ which is uniquely characterized by the property that
\begin{equation}
\label{charac2}
q^*\{f,g\}_Q=\jmath^*\{F, G\}
\end{equation}
for all functions $f,g\in \O_Q$ and all $\rho_M(\c)$-invariant lifts $F,G\in\O_M$ that satisfy $q^*f=\jmath^*F$ and $q^*g=\jmath^*G$.
\end{theorem}
\begin{proof}
It is sufficient to check that the hypotheses of Proposition \ref{reduction} are satisfied. First note that $U$ is a Poisson--Lie subgroup of $G$ and therefore the action of $U$ on $M$ is a Poisson action. 

Since $\Omega$ is transverse to $\mu$,
\[\pi^\#\left(T_{\mu^{-1}(\Omega)}^\circ\right)=\pi^\#\left(\mu^*T_{\Omega}^\circ\right).\]
Suppose that $v$ a local section of this distribution. Then there exists a local 1-form $\alpha$ in $T_\Omega^\circ$ with the property that 
\[v=\pi^\#\mu^*\alpha.\]
Since $\alpha$ is both left- and right-invariant as in the proof of Proposition \ref{sharp}, it is in particular invariant under the action of $B\times_T\nB$. Since the open dense Bruhat cell $B\nB$ is a single $B\times_T\nB$-orbit, this implies that, using the notation of Lemma \ref{forms}, there is some $\xi$ in the subalgebra $\uu+\c$ with the property that $\alpha=\alpha_\xi$. The moment map condition of Definition \ref{defn} then implies that
\[v=\pi^\#\mu^*(\alpha_\xi)=\rho_M(\xi).\]
We conclude that 
\begin{equation}
\label{prechar}
\pi^\#\left(T_{\mu^{-1}(\Omega)}^\circ\right)\subset\rho_M(\uu+\c).
\end{equation}

Moreover, if $v$ is contained in 
\[\pi^\#\left(T_{\mu^{-1}(\Omega)}^\circ\right)\cap T_{\mu^{-1}(\Omega)}=\pi^\#\left(\mu^*T_{\Omega}^\circ\right)\cap \mu_*^{-1}\left(T_\Omega\right),\]
then 
\[\mu_*v=\pi_0^\#(\alpha_\xi)\in T_\Omega\]
and therefore $\xi$ is an element of $\uu$ by Corollary \ref{char}. It follows that
\[\pi^\#\left(T_{\mu^{-1}(\Omega)}^\circ\right)\cap T_{\mu^{-1}(\Omega)}=\pi^\#\left(\mu^*T_{\Omega}^\circ\right)\cap \mu_*^{-1}\left(T_\Omega\right)\subset\rho_M(\uu),\]
so the second condition of Proposition \ref{reduction} is satisfied.

Since the action of $C$ on $\Omega$ is locally free, so is the action of $C$ on the preimage $\mu^{-1}(\Omega)$. This implies that every function on $\mu^{-1}(\Omega)$ locally extends to a function on $M$ which is constant along the distribution $\rho_M(\c)$. In particular, every $U$-invariant function on $\mu^{-1}(\Omega)$ extends locally to a $\rho_M(\uu+\c)$-invariant, and therefore by \eqref{prechar} to a function which is invariant with respect to the distribution
\[\pi^\#\left(T_{\mu^{-1}(\Omega)}^\circ\right).\]
Therefore the third condition of Proposition \ref{reduction} is also met, and the proposition then implies that the quotient $Q$ carries a Poisson structure satisfying the property \eqref{charac2}.
\end{proof}

%
%
%
%
%
%
%
\section{Symplectic leaves of the reduction}
\label{fourth}
In this section we recall some background on Dirac geometry, and we reformulate the main Theorem \ref{main} in the language of Dirac geometry in order to characterize the symplectic leaves of the reduced spaces. In Section \ref{dirac1} we survey the basics of Dirac structures, and for a more detailed exposition we refer to \cite{bur:13}. In Section \ref{dirac3} we discuss our results from the Dirac perspective and show that open dense subsets of the reduced spaces in Theorem \ref{main} can be viewed as Poisson--Dirac submanifolds.

\subsection{Recollections on Dirac geometry}
\label{dirac1}
The \emph{generalized cotangent bundle} of a manifold $M$ is the vector bundle sum $T_M\oplus T^*_M$, equipped with a nondegenerate inner product
\begin{equation}
\label{inner}
\langle(v,\alpha),(w,\beta)\rangle = \alpha(w)+\beta(v)
\end{equation}
and a bracket
\begin{equation}
\label{dorf}
\llbracket(v,\alpha),(w,\beta)\rrbracket =\left([v,w], \mathcal{L}_v\beta-\imath_wd\alpha\right).
\end{equation}
A \emph{Dirac structure} on $M$ is a smooth subbundle $L$ of the generalized cotangent bundle $T_M\oplus T^*_M$ which is 
\begin{itemize}[topsep=2.5pt, itemsep=2.5pt]
\item Langrangian with respect to the inner product \eqref{inner}, and
\item involutive with respect to the bracket \eqref{dorf}. 
\end{itemize}
A simple calculation shows that \eqref{dorf} restricts to a Lie bracket on the sections of $L$, giving $L$ the structure of a Lie algebroid over $M$ with anchor map defined by the first projection.

\begin{example}
If $\omega$ is a 2-form on $M$, its graph
\[L_\omega\coloneqq \left\{(v,\iota_v\omega)\in T_M\oplus T^*_M\mid v\in T_M\right\}\]
is a Lagrangian subbundle of $T_M\oplus T_M^*$ which is involutive if and only if $\omega$ is closed. In particular, the graph of any symplectic form induces a Dirac structure on $M$. Conversely, a Dirac structure $L$ is the graph of a $2$-form if and only if
\begin{equation}
\label{gr1}
L\cap(0\oplus T_M^*)=0,
\end{equation}
and it is the graph of a symplectic form if and only if 
\[L\cap(T_M\oplus 0)=0=L\cap(0\oplus T^*_M).\]
\end{example}

Involutivity under the bracket \eqref{dorf} implies that the image of the first projection $L\too T_M$ is an integrable generalized distribution on $M$, and therefore induces a possibly singular foliation $M=\sqcup\O.$ Any leaf $\imath:\O\htoo M$ of this foliation has an induced Dirac structure
\[\imath^!L\coloneqq\{(v,\imath^*\alpha)\in T_\O\oplus T_\O^*\mid (\imath_*v,\alpha)\in L\}\]
which satisfies condition \eqref{gr1} and is therefore the graph of a $2$-form $\omega_\O\in\Omega^2(\O)$. This form is closed and possibly degenerate, and is called the \emph{presymplectic form} of the leaf $\O$.

\begin{example}
\label{bivectors}
The graph
\[L_\pi\coloneqq \left\{(\pi^\#(\alpha),\alpha)\in T_M\oplus T^*_M\mid \alpha\in T^*_M\right\}\]
of any bivector $\pi$ on $M$ is a Lagrangian subbundle of $T_M\oplus T_M^*$ which is involutive if and only if $\pi$ is Poisson---in other words, if and only if the bracket corresponding to $\pi$ satisfies the Jacobi identity. In this case, the corresponding foliation of $M$ by presymplectic leaves is precisely the usual symplectic foliation induced by the Poisson structure. Conversely, a Dirac structure
\[L\subset T_M\oplus T^*_M\]
is the graph of a bivector only if its kernel
\[\ker L\coloneqq L\cap (T_M\oplus 0)\]
is trivial, and in this case involutivity implies that the corresponding bivector is automatically Poisson. 
\end{example}

Let $(M,L_M)$ and $(N,L_N)$ be Dirac manifolds and let $f:M\too N$ be a smooth map. The \emph{Dirac pullback} of $L_N$ is the Lagrangian distribution
\[f^!L_N\coloneqq\left\{(v,f^*\alpha)\in T_M\oplus T^*_M\mid (f_*v,\alpha)\in L_N\right\}.\]
If $f^!L_N$ is a smooth bundle, it defines a Dirac structure on $N$, and the map $f$ is called \emph{backward-Dirac} or \emph{b-Dirac} if 
\[f^!L_N=L_M.\]
Backward-Dirac maps generalize the pullback of differential forms, and symplectomorphisms, for instance, are always b-Dirac.

The \emph{Dirac pushforward} of $L_M$, if it is well-defined, is the Lagrangian distribution
\[f_!L_M\coloneqq\left\{(f_*v,\alpha)\in T_N\oplus T^*_N\mid (v,f^*\alpha)\in L_M\right\},\]
and $f$ is called \emph{forward-Dirac} or \emph{f-Dirac} if 
\[f_!L_M=L_N.\]
Dirac pushforwards generalize the pushforward of vector fields---Poisson maps, for example, are always f-Dirac.

Note that a b-Dirac map need not be f-Dirac in general, nor the other way around. However, a diffeomorphism is b-Dirac if and only if it is f-Dirac, and in this case it is called a \emph{Dirac diffeomorphism.}

%
%
%
%
%
%
%
\subsection{Poisson--Dirac submanifolds}
\label{dirac3}
Suppose that $(M,\pi)$ is a Poisson manifold and let $L_\pi$ be the associated Dirac structure on $M$. A \emph{Poisson--Dirac submanifold} of $M$ is a submanifold $\imath:X\htoo M$ on which the pullback Dirac structure
\[\imath^!L_\pi\]
is induced by a Poisson bivector as in Example \ref{bivectors}. If, moreover, $X$ intersects the symplectic leaves of $M$ cleanly, then the symplectic foliation induced by $\imath^!L_\pi$ is given by the intersections of $X$ with the symplectic foliation of $M$, and $X$ is called a \emph{clean Poisson--Dirac submanifold} \cite[Lemma 2.10]{pedrodavid}. Poisson submanifolds and Poisson transversals are both special classes of clean Poisson--Dirac submanifolds. 

Let $U_Z$ and $\nuU_Z$ be the positive and negative unipotent radicals of the reductive subgroup $Z$, and let $T_Z=T^w$ be its maximal torus. Then the open dense Bruhat cell of $Z$ is given by
\[Z^\circ\coloneqq U_ZT_Z\nuU_Z,\]
and we can define $U$-stable open dense subsets of our earlier slices by
\[\Sigma^\circ\coloneqq UZ^\circ w\qquad\text{and}\qquad\Omega^\circ\coloneqq UZ^\circ wU.\]
In particular, by \cite[Lemma 2.10]{eve.lu:07}, the open slice $\Sigma^\circ$ is contained in a single $B\times_T\nB$-orbit.

The open slices $\Sigma^\circ$ and $\Omega^\circ$ also satisfy the transversality condition \eqref{transverse2}, and are studied in the work of Duan \cite{dua:23}. Therefore, for any moment map $\mu:M\too G$ there is a commutative diagram
\begin{equation*}
\begin{tikzcd}[row sep=huge]
\mu^{-1}(\Sigma^\circ)\arrow[r, hook, "\imath"]\arrow[rd,swap, "\sim"]	&\mu^{-1}(\Omega)\arrow[r, hook, "\jmath"]\arrow[d, "q"]	&M		\\
						&Q^\circ\coloneqq\mu^{-1}(\Omega^\circ)/U.					&		
\end{tikzcd}
\end{equation*}
Theorem \ref{main} then implies that $Q^\circ$ is a Poisson submanifold of $Q$.

\begin{theorem}
\label{main2}
The preimage $\mu^{-1}(\Sigma^\circ)$ is a clean Poisson--Dirac submanifold of $M$.
\end{theorem}
\begin{proof}
In the language of Dirac structures, the Poisson structure on $Q^\circ$ defined in Theorem \ref{main} corresponds to the Dirac bundle
\[q_!\jmath^{!}L_\pi.\]
The map $q$ is therefore f-Dirac, the inclusion $\imath$ is b-Dirac, and the diagonal morphism is an isomorphism by \eqref{transverse2}. It follows from \cite[Lemma 1.21]{bal:21} that the diagonal morphism is a Dirac diffeomorphism. Therefore, since $q_!\jmath^{!}L_\pi$ is a Poisson structure, so is the pullback $\imath^{!}\jmath^{!}L_\pi$, and $\mu^{-1}(\Sigma^\circ)$ is a Poisson--Dirac submanifold of $M$. In particular, $\Sigma^\circ$ is a Poisson--Dirac submanifold of $G$.

Now we show that $\Sigma^\circ$ is clean. Let $\O_\Delta\subset G$ be a conjugacy class and let $\O$ be the unique $B\times_T\nB$-orbit containing $\Sigma^\circ$. Then
\[T_{\Sigma^\circ}\cap T_{\O_\Delta\cap\O}=T_{\Sigma^\circ}\cap T_{\O_\Delta}\cap T_{\O}=T_{\Sigma^\circ}\cap T_{\O_\Delta}=T_{\Sigma^\circ\cap\O_\Delta}=T_{\Sigma^\circ\cap\O_\Delta\cap \O},\]
where the first equality follows since $\O_\Delta$ and $\O$ intersect transversally, the second follows from the inclusion of $\Sigma^\circ$ into $\O$, and the third follows from the fact that $\Sigma^\circ$ is transverse to the conjugacy classes in $G$ \cite[Theorem 1.2]{dua:23}. Therefore $\Sigma^\circ$ intersects the symplectic leaves of $G$ cleanly.

Now let $S$ be a symplectic leaf in $M$. Then, since $\Sigma^\circ$ is transverse to $\mu$,
\begin{align*}
T_{\mu^{-1}(\Sigma^\circ)}\cap T_{S}=\mu_*^{-1}(T_{\Sigma^\circ})\cap \pi^\#(T_M)
				\subset \mu_*^{-1}&(T_{\Sigma^\circ}\cap \pi_0^\#(T_G))\\	
				&=\mu_*^{-1}(T_{\Sigma^\circ\cap \O_\Delta\cap \O})\\			
				&=T_{\mu^{-1}(\Sigma^\circ\cap \O_\Delta\cap \O)}		
				\subset T_{\mu^{-1}(\Sigma^\circ)\cap S}.
				\end{align*}
It follows that $\mu^{-1}(\Sigma^\circ)$ intersects the symplectic leaves of $M$ cleanly, and therefore it is a clean Poisson--Dirac submanifold of $M$.
\end{proof}

\begin{corollary}
\label{main2cor}
The symplectic leaves of $\mu^{-1}(\Omega^\circ)/U$ are the connected components of the reductions of the symplectic leaves of $M$.
\end{corollary}
\begin{proof}
Suppose that $S$ is a symplectic leaf of $\mu^{-1}(\Omega^\circ)/U$. Since $q\circ\imath$ is an isomorphism, the preimage 
\[S_\Sigma\coloneqq \imath^{-1}(q^{-1}(S))=q^{-1}(S)\cap\mu^{-1}(\Sigma)\]
is a symplectic leaf of $\mu^{-1}(\Sigma)$. Since $\mu^{-1}(\Sigma)$ is a clean Poisson--Dirac submanifold of $M$, by \cite[Lemma 1.2]{pedrodavid} there exists a leaf $S_M$ of $M$ such that $S_\Sigma$ is a connected component of
\[S_M\cap \mu^{-1}(\Sigma).\]
This implies that $S=q(S_\Sigma)$ is a connected component of 
\[q(S_M\cap \mu^{-1}(\Sigma))=q(S_M\cap \mu^{-1}(\Omega)),\]
completing the proof.
\end{proof}

%
%
%
%
%
%
%
\bibliographystyle{plain}
\bibliography{biblio}

\end{document}